\documentclass[a4paper]{amsart}

\usepackage[utf8]{inputenc}

\usepackage{blindtext}
\usepackage{color}

\calclayout

\usepackage[utf8]{inputenc}
\usepackage[T1]{fontenc}
\usepackage{csquotes}
\usepackage[english]{babel}

\usepackage{algpseudocode}
\usepackage{amsfonts}
\usepackage{amssymb} 
\usepackage{mathtools}
\usepackage{mathpartir} 
\usepackage[all,cmtip]{xy}
\usepackage{thmtools}

\usepackage[colorlinks=false,hidelinks]{hyperref}
\usepackage{cleveref}

\usepackage{enumitem}
\setenumerate{label=(\roman*),itemsep=3pt}

\declaretheorem[
name = Theorem,
]{theorem}
\declaretheorem[
name = Theorem,
unnumbered
]{theorem*}
\declaretheorem[
name = Corollary,
sibling = theorem
]{corollary}
\declaretheorem[
name = Lemma,
sibling = theorem
]{lemma}
\declaretheorem[
name = Fact,
sibling = theorem
]{fact}
\declaretheorem[
name = Proposition,
sibling = theorem
]{proposition}

\declaretheorem[
name = Claim,
numbered = no
]{claim*}

\declaretheoremstyle[%
  spaceabove=-6pt,%
  spacebelow=6pt,%
  headfont=\normalfont\itshape,%
  postheadspace=1em,%
  qed=$\blacksquare$,%
  headpunct={.}
]{mystyle}

\declaretheorem[
name = Remark,
sibling = theorem,
style=definition
]{remark}

\declaretheorem[
name = Definition,
sibling = theorem,
style=definition
]{definition}

\newcommand{\IPC}{\mathbf{IPC}}
\newcommand{\IQC}{\mathbf{IQC}}

\newcommand{\vsim}{\mathrel{\scalebox{1}[1.5]{$\shortmid$}\mkern-3.1mu\raisebox{0.15ex}{$\sim$}}}

\newcommand{\set}[1]{\{ #1 \}}
\newcommand{\Set}[2]{\{ #1 \, | \, #2 \}}
\newcommand{\seq}[1]{\langle #1 \rangle}
\newcommand{\Seq}[2]{\langle #1 \, | \, #2 \rangle}
\newcommand{\godel}[1]{\ulcorner #1 \urcorner}
\newcommand{\Ord}{\mathrm{Ord}}

\newcommand{\SRM}{\mathrm{SRM}}

\newcommand{\CZF}{\mathrm{CZF}}
\newcommand{\IZF}{\mathrm{IZF}}
\newcommand{\IKP}{\mathrm{IKP}}
\newcommand{\HA}{\mathrm{HA}}
\newcommand{\ZFC}{\mathrm{ZFC}}

\newcommand{\AC}{\mathrm{AC}}

\newcommand{\Pow}{\mathcal{P}}

\newcommand{\VV}{\mathrm{V}} 
\newcommand{\LL}{\mathrm{L}} 

\newcommand{\id}{\mathrm{id}}

\newcommand{\PL}{\mathbf{PL}}
\newcommand{\QL}{\mathbf{QL}}

\DeclareMathOperator{\dom}{dom}
\DeclareMathOperator{\rank}{rank}

\newcommand{\IFGOTO}[2]{\mathtt{IF} \ #1 \  \mathtt{THEN \ GO \ TO} \ #2}
\newcommand{\GOTO}{\mathtt{GO \ TO} \ }
\newcommand{\TAKE}{\mathtt{TAKE}}
\newcommand{\REMOVE}{\mathtt{REMOVE}}
\newcommand{\ADD}{\mathtt{ADD}}
\newcommand{\COPY}{\mathtt{COPY}}
\newcommand{\POW}{\mathtt{POW}}

\usepackage[square,numbers]{natbib}

\renewcommand{\phi}{\varphi} 

\title{The first-order logic of CZF is intuitionistic first-order logic}
\date{\today}

\author[Robert Passmann]{Robert Passmann}
\address{Institute for Logic, Language and Computation, Faculty of Science, University of Amsterdam, P.O. Box 94242, 1090 GE Amsterdam, The Netherlands}
\email{r.passmann@uva.nl}

\begin{document}

\maketitle

\begin{abstract}
    We prove that the first-order logic of CZF is intuitionistic first-order logic. To do so, we introduce a new model of transfinite computation (Set Register Machines) and combine the resulting notion of realisability with Beth semantics. On the way, we also show that the propositional admissible rules of CZF are exactly those of intuitionistic  propositional logic.
\end{abstract}

\section{Introduction}

The \textit{first-order logic of a theory $T$} consists of those first-order formulas for which all substitution instances are provable in $T$. A classical result of \citet{FriedmanScedrov1986} is that very few axioms suffice for a set theory to exceed the logical strength of intuitionistic first-order logic:

\begin{theorem*}[Friedman \& Ščedrov, 1986]
    Let $T$ be a set theory based on intuitionistic first-order logic that contains the axioms of extensionality, pairing and (finite) union, as well as the separation schema. Then the first-order logic of $T$ exceeds the strength of intuitionistic first-order logic. 
\end{theorem*}

This result applies to intuitionistic Zermelo-Fraenkel Set Theory ($\IZF$) but not to constructive Zermelo-Fraenkel set theory ($\CZF$) because the separation schema of $\CZF$ is restricted to $\Delta_0$-formulas. It has, thus, been a long-standing open question whether the first-order logic of $\CZF$ exceeds the strength of intuitionistic logic as well. We give an answer to this question:

\begin{theorem*}[{see \Cref{Theorem: FOL of CZF is intuitionistic}}]
    The first-order logic of $\CZF$ is intuitionistic first-order logic. 
\end{theorem*}

We prove this result by developing a realisability semantics for $\CZF$ based on a new model of transfinite computation, the so-called \textit{Set Register Machines} ($\SRM$s). Related notions of realisability had earlier been studied by \citet{Rathjen2006} and \citet{Tharp1971}. Our main result is obtained by adapting a technique that \citet{vanOosten1991} had developed for Heyting arithmetic: we combine the resulting notion of $\SRM$-realisability with Beth semantics to obtain a model of $\CZF$ that matches logical truth in a universal Beth model.

\citet{CarlGaleottiPassmann2020} gave a first proof-theoretic application of transfinite computability and provided a realisability interpretation for (infinitary) $\IKP$ set theory using OTMs. In particular, they proved that the propositional admissible rules of $\IKP$ are exactly the admissible rules of intuitionistic propositional logic. On the way to proving our main result, we will prove the same result for $\CZF$. Our motivation for introducing SRMs instead of working with OTMs is that the former are easier adapted for realising stronger set theories than $\IKP$. This work is thus another fruitful application of techniques of transfinite computability to proof-theoretic questions.

\subsection*{Overview}
After recalling some preliminaries in \cref{Section: Preliminaries}, we will begin, in \cref{Section: Set Register Machines}, with introducing our new notion of transfinite machines, the so-called \textit{set register machines} ($\SRM$s). The main result of this section will be a generalisation of a classical result by Kleene and Post about the existence of mutually irreducible degrees of computability. In \cref{Section: Realisability}, we introduce a realisability semantics based on $\SRM$s and show that (a certain extension of) these machines allows to realise $\CZF$ set theory. It also serves as a preparation for \cref{Section: Beth Realisability Models}, in which we will combine our realisability semantics with Beth models to prove our main result. 

\section{Preliminaries}
\label{Section: Preliminaries}

\subsection{Constructive Set Theory} 

We will be concerned with constructive Zermelo-Fraenkel set theory, $\CZF$, and now recall its definition and some basic facts. First, recall the axiom schemes of \textit{strong collection},
$$
    \forall a [ \forall x \in a \exists y \phi(x,y) \rightarrow \exists b (\forall x \in a \exists y \in b \phi(x,y) \wedge \forall y \in b \exists x \in a \phi(x,y))],
$$
for all formulas $\phi$, in which $b$ is not free, and \textit{subset collection},
$$
    \forall a \forall b \exists c \forall u [\forall x \in a \exists y \in b \phi(x,y,u) \rightarrow \exists d \in c (\forall x \in a \exists y \in d \phi(x,y,u) \wedge \forall y \in d \exists x \in a \phi(x,y,u)) ],
$$
for all formulas $\phi$, in which $c$ is not free. By $\Delta_0$-separation we denote the restriction of the separation schema to $\Delta_0$-formulas.

\begin{definition}
    Constructive Zermelo-Fraenkel Set Theory, $\CZF$, is based on intuitionistic first-order logic in the language of set theory and consists of the following axioms and axiom schemes: extensionality, pairing, union, empty set, infinity, $\Delta_0$-separation, strong collection, subset collection, and $\in$-induction.
\end{definition}

We denote $\CZF$ without the subset collection schema by $\CZF^-$. The exponentiation axiom states that function sets exists:
$$
    \forall a \forall b \exists c \forall f (f \in c \leftrightarrow \text{``$f$ is a function from $a$ to $b$''}).
$$
The following is well known (consult, e.g., \citet{AczelRathjen2001}).

\begin{fact}
    \label{Fact: PowerSet implies SubsetCollection}
    In $\CZF^-$, the power set axiom implies the subset collection axiom. Moreover, in $\CZF^-$, the subset collection scheme implies the exponentiation axiom.
\end{fact}

\subsection{Logics \& De Jongh's Theorem} 

Given a theory $T$, based on intuitionistic logic, the logically valid principles of $T$ may exceed those valid in intuitionistic logic. The most well-known example of this phenomen is probably the following consequence of what is known as \textit{Diaconescu's theorem} (see \citet{Diaconescu1975,GoodmanMyhill1978}): $\IZF$ extended with the axiom of choice implies the law of excluded middle, i.e. $\IZF + \AC \vdash \phi \vee \neg \phi$ for all set-theoretic formulas $\phi$. This suggests that it is incorrect to say that the logic of $\IZF + \AC$ is intuitionistic: after all, the law of excluded middle is valid! For this reason, we define the propositional and first-order logics of a theory $T$ as follows, in terms of translations. 

\begin{definition}
    Let $T$ be a theory in a language $\mathcal{L}_T$. A \textit{propositional translation} is a function $\tau$ assigning $\mathcal{L}_T$-sentences to propositional formulas such that:
    \begin{enumerate}
        \item $\tau(p)$ is an $\mathcal{L}_T$-sentence for every propositional letter $p$,
        \item $\tau(\bot) = \bot$, and,
        \item $\tau(A \circ B) = \tau(A) \circ \tau(B)$ for $\circ \in \set{\wedge,\vee,\rightarrow}$.
    \end{enumerate}
    As customary with translations, we will often write $A^\tau$ instead of $\tau(A)$.
\end{definition}

\begin{definition}
    The \textit{propositional logic of $T$}, $\PL(T)$, consists of all propositional formulas $A$ such that $T \vdash A^\tau$ for all propositional translations $\tau$.
\end{definition}

A result concerning the first-order logic of Heyting arithmetic was proved by de Jongh in his doctoral dissertation \cite{DeJonghUnpublished}. We denote intuitionistic propositional logic by $\IPC$ and intuitionistic first-order logic by $\IQC$.

\begin{theorem}[de Jongh, 1970]
    The propositional logic of Heyting arithmetic is intuitionistic propositional logic, $\PL(\HA) = \IPC$.
\end{theorem}

This result is now known as \textit{de Jongh's theorem}, and, in general, we say that a theory $T$ \textit{satisfies de Jongh's theorem} whenever $\PL(T) = \IPC$.

\begin{definition}
    Let $T$ be a theory in a language $\mathcal{L}_T$. A \textit{first-order translation} is a function $\tau$ assigning $\mathcal{L}_T$-formulas to propositional formulas such that:
    \begin{enumerate}
        \item $\tau(R(x_1,\dots,x_n))$ is an $\mathcal{L}_T$-formula $\phi$ with free variables among $x_1,\dots,x_n$,
        \item $\tau(\bot) = \bot$,
        \item $\tau(A \circ B) = \tau(A) \circ \tau(B)$ for $\circ \in \set{\wedge,\vee,\rightarrow}$, and,
        \item $\tau(\mathcal{Q} x A(x)) = \mathcal{Q} x \tau(A(x))$ for $\mathcal{Q} \in \set{\forall,\exists}$.
    \end{enumerate}
\end{definition}

\begin{definition}
    The \textit{first-order logic of $T$}, $\QL(T)$, consists of all first-order formulas $A$ such that $T \vdash A^\tau$ for all first-order translations $\tau$.
\end{definition}

Since de Jongh's initial work, many notable results have been obtained in this area. \citet{Leivant1979} showed that $\QL(\HA) = \IQC$; \citet{vanOosten1991} gave a semantic proof of this fact (the idea of his construction will reappear in our construction in \Cref{Section: Beth Realisability Models}). De Jongh, Verbrugge and Visser \cite{deJonghVerbruggeVisser2011} consider a generalised version of de Jongh's theorem: given a (propositional or first-order) logic $J$ and a theory $T$, we can consider the theory $T(J)$ obtained by closing $T$ under $J$. We then say that $T$ satisfies the \textit{de Jongh property for $J$} if $\PL(T(J)) = J$ (or, $\QL(T(J)) = J$ if $J$ is a first-order logic).

The main negative result concerning logics of set theory is due to \citet{FriedmanScedrov1986}, and was also mentioned in the introduction. Here is a reformulation based on the terminology just introduced.

\begin{theorem}[Friedman \& Ščedrov, 1986]
    Let $T$ be a set theory based on intuitionistic first-order logic that contains the axioms of extensionality, pairing and (finite) union, as well as the separation scheme. Then, $\IQC \subsetneq \QL(T)$.
\end{theorem}

\Citet{Passmann2020} showed that $\PL(\IZF) = \IPC$, and consequently, $\PL(\CZF) = \IPC$. \citet{IemhoffPassmann2021} analysed the logical structure of $\IKP$ and proved, among other things, that $\QL(\IKP) = \IQC$.

\subsection{Admissible Rules} 

We can further generalise our analysis of the logical structure of a given theory by not only considering its logically valid principles but also by determining its admissible rules.

\begin{definition}
    Let $T$ be a theory in a language $\mathcal{L}_T$, and let $A$ and $B$ be propositional formulas. We say that a propositional rule $A / B$ is \textit{admissible in $T$}, written $A \vsim_T B$, if and only if $T \vdash A^\tau$ implies $T \vdash B^\tau$ for all propositional translations $\tau$.
\end{definition}

We say that a theory $T$ has the disjunction property if $T \vdash \phi \vee \psi$ implies $T \vdash \phi$ or $T \vdash \psi$. The \textit{restricted Visser's rules} $\set{V_n}_{n < \omega}$ are defined as follows and play a special role for admissibility (\citet{Iemhoff2001} proved that they form a so-called basis of the admissible rules of intuitionistic propositional logic):
\begin{mathpar}
    \inferrule{
            \left( \bigwedge_{i = 1}^n (p_i \rightarrow q_i) \right) \rightarrow (p_{n+1} \vee p_{n+2})
        }
        {
            \bigvee_{j = 1}^{n+2}\left( \bigwedge_{i = 1}^n (p_i \rightarrow q_i) \rightarrow p_j \right) 
        }
\end{mathpar}
Denote by $V_n^a$ the antecedent and by $V_n^c$ the consequent of the rule.
We will make use of the following result of \citet{Iemhoff2005} to determine admissible rules. 

\begin{theorem}[Iemhoff, {\cite[Theorem 3.9, Corollary 3.10]{Iemhoff2005}}]
    \label{Theorem: Iemhoff}
    If the restricted Visser's rules are propositional admissible for a theory $T$ with the disjunction property, then the propositional admissible rules of $T$ are exactly the propositional admissible rules of intuitionistic propositional logic, ${\vsim_T} = {\vsim_\IPC}$.
\end{theorem}

Visser \cite{Visser1999} proved that the propositional admissible rules of Heyting Arithmetic $\HA$ are exactly the admissible rules of intuitionistic propositional logic $\IPC$. Using realisability techniques, \citet{CarlGaleottiPassmann2020} determined the propositional admissible rules of $\IKP$ to be exactly the admissible rules of propositional intuitionistic logic. \citet{IemhoffPassmannUnpublished2019} proved that the propositional admissible rules of $\CZF_{\mathrm{ER}}$ and $\IZF_{\mathrm{R}}$ are the admissible rules of intuitionistic propositional logic by using a modification of the so-called blended models (earlier introduced by Passmann \cite{Passmann2020}).\footnote{To obtain $\CZF_{\mathrm{ER}}$ and $\IZF_{\mathrm{R}}$, replace subset collection and (strong) collection by exponentiation and replacement, respectively.} It is possible to consider first-order admissible rules; \citet{vandenBergMoerdijk2012} show that certain constructive principles are first-order admissible rules of $\CZF$ (calling them \textit{derived} rules).

\section{Set Register Machines}
\label{Section: Set Register Machines}

\subsection{Definitions \& Basic Properties}

Let us begin with some intuition for set register machines ($\SRM$s). A set register machine has a finite set of registers $R_0,\dots,R_n$ on which it conducts computations. However, the registers do not contain natural numbers (as in the case of register machines) or ordinal numbers (as in the case of ordinal register or Turing machines) but rather arbitrary sets. Accordingly, $\SRM$s use a different set of operations: for example, adding a set contained in a register to another register, or removing a member of a set contained in a certain register. 

We assume that $<_\tau$ is a global well-ordering such that $\rank(x) < \rank(y)$ implies $x <_\tau y$.\footnote{Whenever $<_\tau$ is a global well-ordering, we can assume that this is the case by defining $x <_\tau' y$ if and only if $\rank(x) < \rank(y)$ or $\rank(x) = \rank(y)$ and $x <_\tau y$. Note that $<_\tau'$ is again a well-order.} This means that we are working under the assumption of the global axiom of choice and extend our set-theoretical language with the symbol $<_\tau$. Note that this extended theory is conservative over $\ZFC$ (see \citet[72-73]{Fraenkel1973}). The reason for using this theory as our meta-theory is that we want SRM-computations to be deterministic, and assuming a global well-ordering is a convenient way to achieve this. For a discussion of alternatives see \Cref{Remark: Alternatives for global well-ordering}.

We will now first define programs by giving the permissible operations, and then computations for set register machines. While defining the permissible operations, we will directly give an intuitive description of what the operation does. 

\begin{definition}
    \label{Definition: SRM program}
    A \textit{set register program} $p$ is a finite sequence $p = (p_0, \dots, p_{n-1})$, where each $p_i$ is one of the following commands:
    \begin{enumerate}
        \item ``$R_i := \emptyset$'': replace the content of the $i$th register with the empty set.
        \item ``$\mathtt{ADD}(i,j)$'': replace the content of the $j$th register with $R_j \cup \{ R_i \}$.
        \item ``$\mathtt{COPY}(i,j)$'': replace the content of the $j$th register with $R_i$.
        \item ``$\mathtt{TAKE}(i,j)$'': replace the content of the $j$th register with the $<_\tau$-least set contained in $R_i$, if $R_i$ is non-empty.
        \item ``$\mathtt{REMOVE}(i,j)$'': replace the content of the $j$th register with the set $R_j \setminus \{ R_i \}$.
        \item ``$\mathtt{IF} \ R_i = \emptyset \ \mathtt{THEN \ GO \ TO} \ k$'': check whether the $i$th register is empty; if so, move to program line $k$, and, if not, move to the next line.
        \item ``$\mathtt{IF} \ R_i \in R_j \ \mathtt{THEN \ GO \ TO} \ k$'': check whether $R_i \in R_j$; if so, move to program line $k$, and, if not, move to the next line.
        \item ``$\mathtt{POW}(i,j)$'': replace the content of the $j$th register with the power set of $R_i$.
    \end{enumerate}
\end{definition}

\begin{definition}
    \label{Definition: SRM computation}
    Let $p$ be a set register program and $k < \omega$ be the highest register index appearing in $p$. A \textit{configuration of $p$} is a sequence $(l,r_0,\dots,r_k)$ consisting of the active program line $l < \omega$ and the current content $r_i$ of register $R_i$. If $c = (l,r_0,\dots,r_k)$ is a configuration of $p$, then its successor configuration $c^+ = (l^+,r_0^+,\dots,r_k^+)$ is obtained as follows:
    \begin{enumerate}
        \item If $p_l$ is $\text{``} R_i := \emptyset \text{''}$, then let $r_i^+ = \emptyset$, $r_n^+ = r_n$ for $n \neq i$, and $l^+ = l + 1$.
        \item If $p_l$ is $\text{``} \mathtt{ADD}(i,j) \text{''}$, then let $r_j^+ = r_j \cup \set{r_i}$, $r_n^+ = r_n$ for $n \neq j$, and $l^+ = l + 1$.
        \item If $p_l$ is $\text{``} \mathtt{COPY}(i,j) \text{''}$, then let $r_j^+ = r_i$, $r_n^+ = r_n$ for $n \neq j$, and $l^+ = l + 1$.
        \item If $p_l$ is $\text{``} \mathtt{TAKE}(i,j) \text{''}$, then let $r_j^+$ be the $<_\tau$-minimal element of $r_i$ (if that exists; if $r_i = \emptyset$, then $r_j^+ = r_j$), $r_n^+ = r_n$ for $n \neq j$, and $l^+ = l + 1$.
        \item If $p_l$ is $\text{``} \mathtt{REMOVE}(i,j) \text{''}$, then let $r_j^+ = r_j \setminus \set{r_i}$, $r_n^+ = r_n$ for $n \neq j$, and $l^+ = l + 1$.
        \item If $p_l$ is $\text{``} \mathtt{IF} \ R_i = \emptyset \ \mathtt{THEN \ GO \ TO} \ m \text{''}$, then $r_i^+ = r_i$ for all $i \leq k$; and, if $r_i = \emptyset$, then $l^+ = m$; if $r_i \neq \emptyset$, then $l^+ = l + 1$.
        \item If $p_l$ is $\text{``} \mathtt{IF} \ R_i \in R_j \ \mathtt{THEN \ GO \ TO} \ m \text{''}$, then $r_i^+ = r_i$ for all $i \leq k$; and, if $r_i \in r_j$, then $l^+ = m$; if $r_i \notin r_j$, then $l^+ = l + 1$.
        \item If $p_l$ is ``$\mathtt{POW}(i,j)$'', then $r_j^+ = \Pow(r_i)$, $r_n^+ = r_n$ for all $n \neq i$, and $l^+ = l + 1$.
    \end{enumerate}
    A \textit{computation of $p$ with input $x_0,\dots,x_j$} is a sequence $d$ of ordinal length $\alpha + 1$ consisting of configurations of $p$ such that: 
    \begin{enumerate}
        \item $d_0 = (1,x_0, \dots, x_j,  \emptyset, \dots, \emptyset)$,
        \item if $\beta < \alpha$, then $d_{\beta + 1} = d_\beta^+$,
        \item if $\beta < \alpha$ is a limit, then $l_\beta = \liminf_{\gamma < \beta} l_\gamma$, and $r_\beta = \liminf_{\gamma < \beta} r_\gamma$, where the limes inferior of a sequence of sets is the set obtained from the limes inferior of the characteristic functions, and,
        \item $d_\alpha^+$ is undefined (i.e., $l_\alpha > m$). 
    \end{enumerate}
\end{definition}

The notion of computability obtained by restricting \Cref{Definition: SRM program,Definition: SRM computation} to clauses (i) -- (vii) will be referred to as $\SRM$; the full notion will be referred to as $\SRM^+$. In other words, $\SRM^+$ is obtained from $\SRM$ by adding the power set operation. We allow $\SRM$s and $\SRM^+$s to make use of finitely many set parameters which will be treated as additional input in a fixed register as specified in the program code.

\begin{remark}
    \label{Remark: Alternatives for global well-ordering}
    There are several alternatives for working with a global well-ordering function $<_\tau$: first, it is possible to develop a theory of non-deterministic SRMs, where the $\mathtt{TAKE}$-command takes an arbitrary set. Second, SRMs could work on well-ordered sets (i.e. sets equipped with a well-order). This approach is not useful for $\SRM^+$ as there is no canonical way in extending the well-ordering of a set to its power set (i.e. a certain degree of non-determinateness is introduced again). A third approach is to make computations dependent on a large enough well-ordering of some initial $V_\alpha$. Finally, one could work in the constructible universe $\mathrm{L}$ where we have a $\Sigma_1$-definable well-ordering $<_\mathrm{L}$. We will, in fact, consider this approach in \Cref{Subsection: Restricting to L} but for different reasons: for our main application, we need computations to be definable in the language of set theory without an additional symbol for the global well-ordering.
\end{remark}

\begin{definition}
    A function $f$ is \textit{$\SRM^{(+)}$-computable} if there is an $\SRM^{(+)}$-program $p$, possibly with parameters, which computes $f(x)$ on input $x$. A predicate is called $\SRM^{(+)}$-computable if its characteristic function is $\SRM^{(+)}$-computable.
\end{definition}

Note that every function with set-sized domain is $\SRM$-computable. Clearly, if a function or predicate is $\SRM$-computable, then it is also $\SRM^+$-computable. The converse does not hold: consider, for example, the power set operation.

\begin{proposition}
    Equality of sets is $\SRM$-computable.
\end{proposition}
\begin{proof}
    The following SRM-program computes whether the sets contained in registers $R_0$ and $R_1$ are equal: the program successively takes elements of the first set, checks whether they are contained in the second set, and removes the element from both sets. If both registers $R_0$ and $R_1$ are empty at the same time, then the original sets must have been equal. Otherwise, the original sets were not equal. 
    \begin{quote}
    \begin{algorithmic}[1]
        \State $\mathtt{IF} \ R_0 = \emptyset \ \mathtt{THEN \ GO \ TO} \ 3$
        \State $\mathtt{GO \ TO} \ 5$
        \State $\mathtt{IF} \ R_1 = \emptyset \ \mathtt{THEN \ GO \ TO} \ 11$ 
        \State $\mathtt{GO \ TO} \ 14$
        \State $\mathtt{TAKE}(0,2)$
        \State $\mathtt{REMOVE}(2,0)$
        \State $\mathtt{IF} \ R_2 \in R_1 \ \mathtt{THEN \ GO \ TO} \ 9$
        \State $\mathtt{GO \ TO} \ 14$ 
        \State $\mathtt{REMOVE}(2,1)$
        \State $\mathtt{GO \ TO} \ 1$
        \State $R_0 := \emptyset$
        \State $\mathtt{ADD}(0,0)$
        \State $\mathtt{GO \ TO} \ 15$
        \State $R_0 := \emptyset$ 
    \end{algorithmic}
    \end{quote}
    Note that the operation ``$\mathtt{GO \ TO} \ i$'' is a shortcut for ``$\mathtt{IF} \ R_j = \emptyset \ \mathtt{THEN \ GO \ TO} \ i$'' where $j$ is chosen in such a way that the register $R_j$ is not mentioned in any other instruction of the program.
\end{proof}

In view of this proposition, we can use an operation ``$\mathtt{IF} \ R_i = R_j \ \mathtt{THEN \ GO \ TO} \ k$'' by implementing the program of the proof of the proposition as a subroutine. The following lemma shows that many basic operations and predicates are $\SRM^+$-computable. 

\begin{lemma}
    \label{Lemma: basic SRM computable functions and predicates}
    The following functions and predicates are $\SRM^+$-computable:
    \begin{enumerate}
        \item the binary union function $(x,y) \mapsto x \cup y$,
        \item the intersection function $(x,y) \mapsto x \cap y$,
        \item the singleton and pairing functions, $x \mapsto \set{x}$ and $(x,y) \mapsto \set{x,y}$,
        \item the ordered pairing function $(x,y) \mapsto \langle x,y \rangle$,
        \item the first and second projections $\langle x,y \rangle \mapsto x$, $\langle x,y \rangle \mapsto y$,
        \item the predicate ``$x$ is an ordered pair'',
        \item the predicate ``$x$ is a function'',
        \item the union of a set, $x \mapsto \bigcup x$,
        \item the intersection of a set, $x \mapsto \bigcap x$,
        \item the function mapping a function to its domain $f \mapsto \dom(f)$,
        \item function application $(f, x) \mapsto f(x)$,
        \item the predicate ``$x$ is an ordinal'',
        \item the predicate ``$x$ is a sequence of ordinal length'',
        \item the function computing the $<_\tau$-least element $x \in y$ satisfying an $\SRM^+$-computable predicate $P(x)$,
        \item the $\alpha$th projection on a sequence, $\Seq{x_i}{i < \beta} \mapsto x_\alpha$,
        \item the power set function, $x \mapsto \Pow(x)$,
        \item the predicate ``$x$ is the power set of $y$'',
        \item the limes inferior of a sequence of sets.
    \end{enumerate}
\end{lemma}
\begin{proof}
    We will give explicit programs for the first few cases and then move to increasingly abstract descriptions of the desired programs:
    \begin{enumerate}
        \item Observe that the following program computes the union of the sets in registers $R_0$ and $R_1$ by adding all elements of $R_1$ to $R_0$:
    \begin{quote}
        \begin{algorithmic}[1]
            \State $\mathtt{IF} \ R_1 = \emptyset \  \mathtt{THEN \ GO \ TO} \ 6$
            \State $\mathtt{TAKE}(1,2)$
            \State $\mathtt{REMOVE}(2,1)$
            \State $\mathtt{ADD}(2,0)$
            \State $\mathtt{GO \ TO} \ 1$
        \end{algorithmic}
    \end{quote}
    
        \item Observe that the intersection of the sets contained in registers $R_0$ and $R_1$ can be computed as follows. Check for each element of $R_1$ whether it is contained in $R_0$ and, if so, save it into a register for the intersection:
        \begin{quote}
            \begin{algorithmic}[1]
                \State $\IFGOTO{R_1 =     \emptyset}{8}$
                \State $\TAKE(1,2)$
                \State $\REMOVE(2,1)$
                \State $\IFGOTO{R_2 \in     R_0}{6}$
                \State $\GOTO 1$
                \State $\ADD(2,3)$
                \State $\GOTO 1$
                \State $\COPY(3,0)$
            \end{algorithmic}
        \end{quote}
    
    \item The functions of (iii) can be easily implemented. 
    
    \item Recall that $\langle x, y \rangle = \set{\set{x},\set{x,y}}$, and this can easily be computed. 
    
    \item Note that $\bigcap \langle x, y \rangle = x$ and $\bigcup \langle x, y \rangle = \set{x,y}$. So we can construct the desired programs by combining the procedures from (i) and (ii) in a straightforward way. 
    
    \item We have to implement a procedure that checks whether $x$ is an ordered pair: use (v) to compute the first and second projection of $x$, say, $y$ and $z$. Then compute $\seq{y,z}$ with (iv) and check whether this equals $x$.
    
    \item Check whether $x$ consists of ordered pairs (using (vi)), and then check that $x$ is functional with (v).
    
    \item Use four registers: $R_0$ contains $x$, $R_1$ for the union of $x$, and $R_2$ and $R_3$ as auxiliary registers. Then proceed as follows: as long as $R_0$ is non-empty, take a set from $R_0$ and save it in $R_2$, then remove it from $R_0$. Then, as long as $R_2$ is non-empty, take an element of $R_2$ and save it in $R_3$, then remove it from $R_2$ and add it to $R_1$. Once $R_0$ is empty, we are done: copy our result from $R_1$ to $R_0$, and stop. 
    
    \item A similar procedure as in the previous item does the job.
    
    \item Take and remove elements from $R_0$ as long as it is non-empty. To each element, apply the first-projection from (v), and add it to $R_1$. Once $R_0$ is empty, $R_1$ contains the domain of $x$. 
    
    \item Search through $f$ until a pair with first coordinate $x$ is found. Then return the second projection of that pair.
    
    \item Observe that it is straightforward to compute whether ``$x$ is a transitive set of transitive sets''.
    
    \item Check whether $x$ is a function whose domain is an ordinal.
    
    \item Given a procedure for checking $P$, take and remove elements from $y$ until some $x$ is found satisfying $P(x)$. By the definition of the $\mathtt{TAKE}$-operation, this will be the $<_\tau$-minimal element of $y$ satisfying $P$.
    
    \item This is just function application.
    
    \item This is straightforward using the $\mathtt{POW}$-operation.
    
    \item Again, straightforward using the $\mathtt{POW}$-operation.
    
    \item Note that the limes inferior of a sequence of sets can be presented as follows:
    $$
        \liminf_{\gamma < \alpha} x_\gamma = \bigcup_{\beta < \alpha} \, \bigcap_{\gamma \in [\beta+1,\alpha)} x_\gamma.
    $$
    This can be straightforwardly implemented by combining the previous items of this lemma.
    \end{enumerate}
\end{proof}

\begin{lemma}
    \label{Lemma: Delta_0 truth decidable}
    Let $\phi(\bar x)$ be a $\Delta_0$-formula. There is an $\SRM$ $p$ such that $p(\godel{\phi},\bar x) = 1$ if $\VV \vDash \phi(\bar x)$ and $p(\godel{\phi},\bar x) =0$ if $\VV \vDash \neg \phi(\bar x)$.
\end{lemma}
\begin{proof}
    We construct a machine that recursively calls itself. For the base cases, let $p(\godel{x_i = x_j},\bar x)$ be the program that returns $1$ if $x_i = x_j$ and $0$ if $x_i \neq x_j$. Similarly, let $p(\godel{x_i \in x_j},\bar x)$ be the program that returns $1$ if $x_i \in x_j$ and $0$ if $x_i \notin x_j$. The cases for conjunction, disjunction and implication are easily constructed by recursion. For the bounded existential quantifier, $\exists x \in a \, \phi(x)$, the machine $p$ conducts a search through $a$ by consecutively taking and removing elements. If $p$ finds some $b \in a$ such that $p(\godel{\phi},\seq{b,a,x})= 1$, then $p$ returns $1$. If no such $b$ is found, then $a$ does not contain a witness for $\phi$ and $p$ returns $0$. The bounded universal quantifier can be implemented similarly with a bounded search. 
\end{proof}

The next theorem shows that moving from Ordinal Turing Machines to Set Register Machines does not increase the computational strength. We do not give a detailed proof since the result is not used in the remainder of this article. 

\begin{theorem}
    Ordinal Turing machines with parameters (OTMs) and set register machines with parameters (SRMs) can simulate each other.
\end{theorem}
\begin{proof}
    For the first direction, recall that OTMs and ordinal register machines (ORMs) can simulate each other (e.g. \citet{Carl2020}). It will, therefore, be enough to show that SRMs simulate ORMs but, in fact, more is true: it is straightforward to see that every ORM-program can be executed by an SRM.
    
    The other direction can be shown by a straightforward but tedious coding argument by using a large enough fragment of the well-order $<_\tau$ as a parameter (\citet{CarlGaleottiPassmann2020} spell out a very similar argument in an appendix; \citet[Section 2.3.2 and Chapter 3]{Carl2020} discusses codings as well).
\end{proof}

\subsection{Oracles and relative computability}

As with other notions of computability, we can enrich $\SRM^+$s with oracles. Let $O: \VV \to \VV$ be a partial class function. We obtain oracle $\SRM^{+,O}$ by extending \Cref{Definition: SRM program} with the following operation:

\begin{quote}
    ``$\mathtt{ORACLE}(i,j)$'': replace the contents of the $j$th register with the result of querying the oracle $O$ with $R_i$.
\end{quote}
We also extend \Cref{Definition: SRM computation}:
\begin{quote}
    If $p_l$ is ``$\mathtt{ORACLE}(i,j)$'', proceed as follows: if $O(r_i)$ is defined, let $r_j^+ = O(r_i)$, $r_n^+ = r_n$ for all $n \neq i$, and $l^+ = l + 1$. If $O(r_i)$ is undefined, let $r_j^+ = r_j$ for all $j \leq k$ and $l^+ = l$.
\end{quote}
The evaluation function is chosen like this to ensure that any $\SRM^{+,O}$ loops whenever the oracle is queried on undefined input. This entails that the oracle is only queried on its domain within a successful computation. Given oracles, we can define a relative notion of computability.

\begin{definition}
    We say that a function $f$ is \textit{$\SRM^+$-computable in $g$} if and only if there is an $\SRM^{+,g}$ program $p$ that computes $f$.
\end{definition}

A function is \textit{$\SRM^+$-computable} if and only if it is $\SRM^+$-computable in the empty function. In fact, a function is \textit{$\SRM^+$-computable} if and only if it is $\SRM^+$-computable in any set-sized function.


We will now work towards generalising a result of \citet{KleenePost1954}, which will be useful later but is also interesting in its own regard. 

\begin{proposition}
    The class function $V_{(\cdot)}: \Ord \to V, \alpha \mapsto V_\alpha$ is $\SRM^+$-computable.
\end{proposition}
\begin{proof}
    An $\SRM^+$-program does this by starting with the empty set and consecutively computing power sets while keeping the current rank in an auxiliary register. The program keeps computing until it reaches the desired $\alpha$. 
    
    This procedure is implemented in the following program, where the input $\alpha$ is written into $R_0$; note that the initial configuration of all other registers is $\emptyset$. We use $R_1$ to count our current stage $\beta$ and $R_2$ to save the current $V_\beta$.
    \begin{quote}
    \begin{algorithmic}[1]
        \State $\mathtt{IF} \ R_0 = R_1 \ \mathtt{THEN \ GO \ TO} \ 5$ 
        \State $\mathtt{POW}(2,2)$
        \State $\mathtt{ADD}(1,1)$
        \State $\mathtt{GO \ TO} \ 1$
    \end{algorithmic}
    \end{quote}
    Note that the register $R_0$ remains unchanged, and the registers $R_1$ and $R_2$ are monotonically increasing. Therefore, the program does the job also at limit stages.
\end{proof}

The following proposition can be anticipated from how the evaluation of the $\mathtt{TAKE}$-operation was defined.

\begin{proposition}
    The global well-ordering $<_\tau$ is $\SRM^+$-decidable.
\end{proposition}
\begin{proof}
    This is implemented by an $\SRM^+$ that does the following: given $a$ and $b$, check whether $a = b$. If so, we are done. If not, compute $\set{a,b}$ and use the $\mathtt{TAKE}$-operation to take a set $c \in \set{a,b}$. By the definition of the $\mathtt{TAKE}$-operation, either $c = a$ and then $a <_\tau b$, or $c = b$ and then $b <_\tau a$.
\end{proof}

By the \emph{$\alpha$th element of $V$ according to $<_\tau$}, we denote the unique $x$ such that the order type of $(\Set{y}{y <_\tau x}, <_\tau)$ is $\alpha$.

\begin{proposition}
    \label{Proposition: V re}
    The bijective class function $F_\tau: \Ord \to V$ mapping $\alpha$ to the $\alpha$th element of $V$ according to $<_\tau$ is $\SRM^+$-computable and so is its inverse.
\end{proposition}
\begin{proof}
    Recall our assumption that $\rank(x) < \rank(y)$ implies $x <_\tau y$. Therefore, computing $<_\tau$ on some $V_\alpha$ means to compute an initial segment of $<_\tau$. We can therefore proceed as follows. 
    
    For the forward direction, use the $\mathtt{POW}$-operation to compute $V_{\alpha+1}$. Then take and remove elements from $V_{\alpha+1}$ while running a counter until it reaches $\alpha$. The last element taken is the set we were looking for.
    
    For the other direction, given $a \in V$, compute a $V_\beta$ such that $a \in V_\beta$. Then start a counter and successively take and remove elements from $V_\beta$ until $a$ is reached. The value of the counter is the ordinal $\alpha$ we are looking for.
\end{proof}

\begin{proposition}
    \label{Proposition: SRM Halting problem}
    Let $O$ be a (partial) class function. The $\SRM^{+,O}$ halting problem is $\SRM^{+,O}$ undecidable. 
\end{proposition}
\begin{proof}
    This is proved by contradiction with the usual diagonal argument. Assume that there is a machine $p$ such that $p(x) = 1$ if and only if $x$ is an $\SRM^+$ that halts, and $p(x) = 0$ otherwise. Then define a machine $q$ such that $q(x)$ does not halt if and only if $p(x) = 1$. Then, $p(q) = 1$ if and only if $q(q)$ does not halt if and only if $p(q) = 0$. A contradiction.
\end{proof}

\begin{proposition}
    Let $O$ be a (partial) class function. Then there is an oracle $\tilde O$ such that there is an $\SRM^{+,\tilde O}$-program $u$ which is universal for $\SRM^{+,O}$, i.e. $u(p,x)$ and $p(x)$ are both defined and equal whenever at least one of them is defined. Moreover, there is an $\SRM^{+,\tilde O}$-program $c$ such that $c(p,x) = 1$ if $x$ is a successful computation of $p$ and $c(p,x) = 0$ otherwise. In particular, if $O$ is the empty function, then $\tilde O$ can be taken empty as well.
\end{proposition}
\begin{proof}
    Let $\tilde O$ be the function such that $\tilde O(x) = \seq{1,O(x)}$ whenever $O(x)$ is defined and $\tilde O(x) = \seq{0,0}$ whenever $O(x)$ is undefined. 
    Using \Cref{Lemma: basic SRM computable functions and predicates} and $\tilde O$, it is straightforward (but tedious) to construct a program $c$ such that $c(p,x) = 1$ if $x$ is a successful computation of $p$ and $c(p,x) = 0$ otherwise. Then note that $p(x)$ is defined if and only if there is a successful computation of $p$ on input $x$. For this reason, the universal machine can be implemented as an unbounded search through $\VV$ that stops if a successful computation for $p$ on input $x$ is found, and returns $p(x)$. 
    In the case where $O$ is the empty function, we can take $\tilde O$ to be the empty function as well because all $\SRM^+$-operations are $\SRM^+$-decidable.
\end{proof}

It is possible to construct an $\SRM^{+,O}$-universal machine for $\SRM^{+,O}$, if one changes the definition of oracle evaluation in such a way that the universal machine can query the oracle without the risk of not halting.

Let $D(x,y)$ be a binary predicate in the language of set theory. Adapting from \citet{KleenePost1954}, we write $D_z(x) := D(x,z)$ and define $D^z$ to be the join of all $D_{y}$ with $y \neq z$, as follows:
$$
    D^z(x,y) := \begin{cases}
        D(x,y), & \text{ if } y \neq z, \\
        0, & \text{ if } y = z.
    \end{cases}
$$
The proof of the following theorem is a generalisation of a result by \citet[Theorem 2]{KleenePost1954}; our proof will be a generalisation of their diagonal argument to the case of $\SRM^+$.

\begin{theorem}
    \label{Theorem: D for SRM^+}
    There is a set-theoretic predicate $D(x,y)$ such that $D_z$ is not $\SRM^+$-computable in $D^z$.
\end{theorem}
\begin{proof}
    We define the predicate by informally describing a total $\SRM^{+,H}$-program that makes use of an oracle $H$ for the $\SRM^+$-halting problem.
    
    Let $R_{init}$ be an auxiliary register which is used to save an initial segment of the predicate we are defining. Let $R_{stage}$ be an auxiliary register that contains an ordinal representing the current stage of the construction.
    
    To ensure the non-computability desired in the theorem, we have to satisfy class-many conditions, for each $\SRM^+$-program $e$ (possibly with parameters) and set $z$:
    \begin{align*}
        \tag{$P_{e,z}$} \text{The program } e \text{ does not witness that } D_z \text{ is $\SRM^+$-computable in } D^z.
    \end{align*}
    Apply the inverse Gödel pairing function to $R_{stage}$ obtain ordinals $\alpha$ and $\beta$. By \Cref{Proposition: V re}, calculate $e := F_\tau^{-1}(\alpha)$ and $z := F_\tau^{-1}(\beta)$. We want to extend $R_{init}$ in such a way that $P_{e,z}$ will hold. To this end, let $x$ be the $<_\tau$-least set for which $R_{init}(x,z)$ is undefined. For convenience, let us say that $E$ is a \textit{$z$-extension of $R_{init}$} if $R_{init} \subseteq E$ and if $R_{init}(w,z)$ is undefined for some $w$ then so is $E(w,z)$. There are two cases to consider.
    
    \textbf{Case 1:} There is a $z$-extension $D_{init}$ of $R_{init}$ such that there is a successful computation of $e$ on input $(x,z)$ using $D_{init}^z$ as an oracle, i.e. the oracle is the predicate obtained from $D_{init}$ by taking $D_{init}^z(w,y) = D_{init}(w,y)$ if $y \neq z$, and $D_{init}^z(w,z) = 0$ for all $w$. Note that our machine can decide whether such an extension exists by using the oracle for the $\SRM^+$-halting problem. Let $y \in \set{0,1}$ be the result of this computation. As $D_{init}$ is a $z$-extension of $R_{init}$, it must be that $D_{init}(x,z)$ is undefined. We can therefore set $R_{init} := D_{init} \cup \set{((x,z), 1 - y)}$. This choice ensures that $e$ does not witness that $D_z$ is computable in $D^z$.
    
    \textbf{Case 2:} For all $z$-extensions $D_{init}$ of $R_{init}$ there is no successful computation of $e$ on input $(x,z)$ with $D_{init}^z$ as oracle. In this case, we let $R_{init} := R_{init} \cup \set{((x,y),0)}$. This (arbitrary) choice works because the final predicate $D$ will be such that there is no successful computation of $e$ on input $(x,z)$ with oracle $D^z$: for contradiction, suppose there was such a successful computation $c$ and consider the $z$-extension $D_{init}$ of $R_{init}$ given by $D_{init}(x,y) = D(x,y)$ for all $(x,y)$, $y \neq z$, for which the oracle is called during the computation $c$. As $D_{init}^z(w,z)$ is defined for all $w$, all oracle calls during the computation $c$ are still the same when using $D_{init}^z$ instead of $D^z$. Hence, there is a successful computation $c$ of $e$ on input $(x,z)$ with oracle $D_{init}^z$. But that is in contradiction to the assumption of this case. 
    
    The program defined this way will eventually give rise to a completely defined predicate $D$ on $V \times V$. The value of $D(x,y)$ can be computed by running the procedure above until the value for $(x,y)$ is known.
\end{proof}

Note that the program described in the proof above does not use any parameters and can thus be coded as a natural number.

\begin{remark}
    In fact, Kleene and Post prove a stronger result which allows to locate $D$ between any two Turing degrees. A similar result is possible here but we leave the proof to the interested reader as we do not need it.
\end{remark}

\subsection{Constructible SRMs}
\label{Subsection: Restricting to L}

For our applications to the first-order logic of CZF, it will be important that we can express the predicate ``$D(x,y)$ holds'' in a way that only uses the language of set theory without introducing an extra relation symbol into our language to refer to the global well-order. This means that we have to circumvent referring to $<_\tau$ as this is an extra symbol that cannot be defined in terms of a set-theoretic formula. Due to the following well-known fact, we will restrict our attention to constructible sets (for reference see, e.g., \citet[Theorem 13.18 \& Lemma 13.19]{Jech2003}):

\begin{fact}
    There is a $\Sigma_1$-definable well-ordering $<_\LL$ of the constructible universe $\LL$.
\end{fact}

So if we restrict our attention to $\SRM^+$s that work only on constructible sets, we can replace $<_\tau$ with $<_\LL$ in \Cref{Definition: SRM computation}. The resulting notion of $\SRM$ will be called \textit{constructible $\SRM^+$} and denoted, in short, by $\SRM^+_\LL$. Note that all of the results obtained so far about $\SRM^+$s can be relativised to $\LL$ and thus transferred to $\SRM^+_\LL$. In particular, we get the following versions of \Cref{Lemma: Delta_0 truth decidable} and \Cref{Theorem: D for SRM^+}:

\begin{lemma}
    \label{Lemma: L Delta_0 truth decidable}
    Let $\phi(\bar x)$ be a $\Delta_0$-formula. There is an $\SRM_\LL$-program $p$ such that $p(\godel{\phi},\bar x) = 1$ if $L \vDash \phi$ and $p(\godel{\phi},\bar x) =0$ if $L \vDash \neg \phi$.
\end{lemma}

\begin{corollary}
    \label{Theorem: D for SRM^+_L}
    There is a non-$\SRM^+_\LL$-computable set-theoretic predicate $D(x,y)$, expressible in the language of set theory, such that $D_z$ is not $\SRM^+_\LL$-computable in $D^z$.
\end{corollary}

\section{Realisability}
\label{Section: Realisability}

We will now define a notion of realisability based on $\SRM^+$s, and observe a few proof-theoretic consequences for $\CZF$.

\begin{definition}
    We define the realisability relation $\Vdash$ recursively for an $\SRM^{(+),(O)}_{(\LL)}$ $r$ as follows:
    \begin{enumerate}
        \item $r \Vdash a \in b$ if and only if $a \in b$;
        \item $r \Vdash a = b$ if and only if $a = b$;
        \item $r \Vdash \phi_0 \wedge \phi_1$ if and only if $r(0) \Vdash \phi_0$ and $r(1) \Vdash \phi_1$;
        \item $r \Vdash \phi_0 \vee \phi_1$ if and only if $r(1) \Vdash \phi_{r(0)}$;
        \item $r \Vdash \phi_0 \rightarrow \phi_1$ if and only if whenever $s \Vdash \phi_0$, then $r(s) \Vdash \phi_1$;
        \item $r \Vdash \exists x \phi(x)$ if and only if $r(1) \Vdash \phi(r(0))$;
        \item $r \Vdash \forall x \phi(x)$ if and only if $r(a) \Vdash \phi(a)$ for every set $a$.
    \end{enumerate}
    We say that $\phi$ is $\SRM$-realisable if and only if there is an $\SRM$ realising $\phi$. Similarly, we say that $\phi$ is $\SRM^+$-realisable if and only if there is an $\SRM^+$ realising $\phi$; and so for $\SRM^{+,O}$, $\SRM^+_\LL$, and $\SRM^{+,O}_\LL$.
\end{definition}

This could be extended to infinitary languages as done by \citet{CarlGaleottiPassmann2020}. Analogously to (i) and (ii), one could give realisability semantics to the global well-order $<_\tau$. 

\begin{theorem}
    $\SRM^{(+),(O)}_{(\LL)}$-realisability is sound for intuitionistic logic.
\end{theorem}
\begin{proof}
    This is a standard argument and can be established, for example, by providing a realiser for every axiom in a Hilbert-style formalisation of $\IQC$ and showing that \textit{modus ponens} is valid. The latter follows immediately from the definition of the relisability relation.
\end{proof}

\begin{lemma}
    \label{Lemma: Sigma_1 realisability is truth}
    Let $\phi(\bar x)$ be a $\Sigma_1$-formula. Then there is some realiser $r \Vdash \phi(\bar x)$ if and only if $\VV \vDash \phi(\bar x)$.
\end{lemma}
\begin{proof}
    This is a straightforward induction on $\Sigma_1$-formulas. We will prove a more intricate version of this lemma below, see \Cref{Lemma: Sigma_1 Beth-realisability is truth in L}.
\end{proof}

\begin{theorem}
    \label{Theorem: CZF realisable}
    The axioms (and schemes) of extensionality, pairing, union, infinity, collection, $\in$-induction, and $\Delta_0$-separation are $\SRM$-realisable. The axiom of choice, $\mathrm{AC}$, is $\SRM$-realisable. The axioms of power set and strong collection are $\SRM^+$-realisable.
    In conclusion, $\IKP + \mathrm{AC}$ is $\SRM$-realisable, and $\CZF + \mathrm{PowerSet} + \mathrm{AC}$ is $\SRM^+$-realisable. Moreover, $\IKP + \mathrm{AC}$ is $\SRM_\LL$-realisable, and $\CZF + \mathrm{PowerSet} + \mathrm{AC}$ is $\SRM^+_\LL$-realisable.
\end{theorem}
\begin{proof}
    It is straightforward to construct a realiser for the extensionality axiom.  
    For the empty set axiom, let $r$ be an $\SRM$ that returns the empty set on input $0$ and the identity function on input $1$. Then $r(1) \Vdash \forall y (y \in r(0) \rightarrow \bot)$ because $\not \Vdash_w y \in \emptyset$ for all $w \in P$ and $y \in \VV$. Hence, $r \Vdash \exists x \forall y (y \notin x)$.
    A realiser for the union axiom is an $\SRM$ $r$ such that, for every $a \in \VV$, $r(a)(0) = \bigcup a$, using \Cref{Lemma: basic SRM computable functions and predicates}, $r(a)(1)(x)(0) = \id$, and $r(a)(1)(x)(1) = \id$ for every $x$. The infinity axiom is realised by an $\SRM$ $r$ with $r(0) = \omega$, $r(1)(x)(0) = \id$, and $r(1)(x)(1) = \id$ for every $x \in \VV$. Using the power set operation provided by $\SRM^+$-programs, it is straightforward to construct a realiser of the power set axiom. Note that the subset collection schema is a consequence of the power set axiom.
    
    Let us consider $\Delta_0$-separation next, i.e. the schema consisting of
    $$
        \forall x \exists y \forall z (z \in y \leftrightarrow z \in x \wedge \phi(x)),
    $$
    where $\phi(x)$ is a $\Delta_0$-formula. By combining \Cref{Lemma: L Delta_0 truth decidable,Lemma: Sigma_1 realisability is truth}, we know that $\Vdash \phi(x)$ if and only if $p(\godel{\phi},x) = 1$, and $p(\godel{\phi},x) = 0$ in case $\not \Vdash \phi(x)$. Hence, we can compute the witnessing set $y$ by conducting a bounded search through $x$ and collecting all $z \in x$ such that $p(\godel{\phi},z) = 1$. It is then trivial to realise $\forall z (z \in y \leftrightarrow z \in x \wedge \phi(x))$ because $\phi$ is a $\Delta_0$-formula.
    
    Consider the schema of $\in$-induction next:
    $$
        \forall x (\forall y \in x \phi(z) \rightarrow \phi(x)) \rightarrow \forall x \phi(x).
    $$
    An $\SRM$ $r$ is a realiser for this if and only if, if $s \Vdash \forall x (\forall y \in x \phi(z) \rightarrow \phi(x))$, then $r(s) \Vdash \forall x \phi(x)$. Now, in this situation, $s$ allows us to iteratively construct realisers for every $x \in \VV$ by successively building realisers for every $\VV_\alpha$. Hence, given $x \in \VV$, we just compute realisers until we reach $x$ and then output the realiser for $\phi(x)$.
    
    Next, we consider the strong collection schema:
    $$
        \forall x [(\forall y \in x \exists z \phi(y,z)) \rightarrow \exists w (\forall y \in x \exists z \in w \phi(y,z) \wedge \forall z \in w \exists y \in x \phi(y,z))],
    $$
    for all formulas $\phi(x,y)$ for which $w$ is not free. Given $x \in \VV$, let $r(x)(s)$, for $s \Vdash \forall y \in x \exists z \phi(z,y)$, be an $\SRM$ that computes a set consisting of all $s(y)(0)$ for every $y \in x$, and returns this set on input $0$. Using $s$, it is straightforward to construct a realiser $r(x)(s)(1) \Vdash \forall y \in x \exists z \in r(x)(s)(0) \ \phi(y,z) \wedge \forall z \in r(x)(s)(0) \ \exists y \in x \phi(y,z))$. 
    
    Finally, consider the axiom of choice,
    $$
        \forall x ((\forall y \in x \exists z \ z \in y) \rightarrow \exists f \forall y \in x \ f(y) \in y).
    $$
    This axiom states that whenever $x$ consists of non-empty sets, then there is a choice function $f$ on $x$. Using \Cref{Lemma: basic SRM computable functions and predicates}, it is straightforward to construct an $\SRM$ that computes such a choice function: for every element of $y \in x$, use the $\mathtt{TAKE}$-operation to obtain some $z \in y$. Then add $(x,y)$ to the register in which we build the choice function. 
    
    The corresponding results for $\SRM_\LL$ and $\SRM^+_\LL$ are obtained through relativisation and absoluteness properties (or by observing that the exact same realisers still do the job).
\end{proof}

It turns out that $\IZF$ is not $\SRM^+$-realisable.

\begin{theorem}
    There is an instance of the separation axiom that is not $\SRM^+$-realisable. In conclusion, $\IZF$ is not $\SRM^+$-realisable.
\end{theorem}
\begin{proof}
    Consider the predicate $H(x,y)$ expressing that ``$x$ is an $\SRM^+$ that halts on input $y$''. One can easily construct a formula $\phi(x,y)$ such that $\phi(x,y)$ is realised if and only if $H(x,y)$ is true (see also the proof of \Cref{Lemma: negative formula for D} for a similar argument). Then let $s$ be a realiser of the following instance of the separation axiom:
    $$
        \forall x \forall y \forall z \exists w \forall u (u \in w \leftrightarrow (u \in z \wedge \phi(x,y))).
    $$
    We can then construct an $\SRM$ $r$ that does the following. Given $x$ and $y$, compute $w := s(x)(y)(1)(0)$ and return the result. By construction, $r(x,y) = 1$ just in case $H(x,y)$ holds, and $r(x,y) = 0$ otherwise. So $r$ is an $\SRM^+$ solving the $\SRM^+$ halting problem but this is impossible, see \Cref{Proposition: SRM Halting problem}.
\end{proof}

In fact, we have just seen that $\CZF + \mathrm{PowerSet}$ is $\SRM^+$-realisable. The following proposition shows that we cannot be more fine-grained: if there is an $\SRM$ realising the exponentiation axiom (possibly using an oracle), then we can already compute power sets. Recall that the axiom of exponentiation is a consequence of subset collection (\Cref{Fact: PowerSet implies SubsetCollection}).

\begin{proposition}
    Let $r$ be an $\SRM$, possibly using an oracle, such that $r$ realises the axiom of exponentiation, then there is an $\SRM$, using $r$ as an oracle, that computes power sets. 
\end{proposition}
\begin{proof}
    Let $r$ be a realiser of the axiom of exponentiation:
    $$
        \forall x \forall y \exists z \forall f (f \in z \leftrightarrow \text{``$f$ is a function from $x$ to $y$''}),
    $$
    where ``$f$ is a function from $x$ to $y$'' is expressed as a $\Delta_0$-formula. Then, given a set $a$, the set $b := r(a)(\set{0,1})(0)$ contains all $f$ for which there is a realiser of ``$f$ is a function from $x$ to $y$''. As this is a $\Delta_0$-formula, \Cref{Lemma: Sigma_1 realisability is truth} implies that $b$ consists of all functions from $a$ to $2$. It is now easy to compute the power set of $a$ as follows: for each element $f$ of $b$, compute the set consisting of exactly those $x \in a$ for which $f(a) = 1$. This results in the power set of $a$ because each subset of $a$ gives rise to its characteristic function contained in $b$.
\end{proof}

Our realisability semantics also allow to give an upper bound for $\Pi_2$-formulas provable in $\CZF$ in terms of the computable strength of $\SRM^+$.

\begin{theorem}
    Let $\phi$ be a $\Sigma_1$-formula. If $\CZF \vdash \forall x \exists y \phi(x,y)$, then there is an $\SRM^+$ $p$ such that $\VV \vDash \phi(x,p(x))$.
\end{theorem}
\begin{proof}
    If $\CZF \vdash \forall x \exists y \phi(x,y)$, then, by \Cref{Theorem: CZF realisable}, there exists an $\SRM^+$ $r \Vdash \forall x \exists y \phi(x,y)$. Take $p(x)$ to be the $\SRM^+$ to compute $r(x)(0)$. Then, for all $x$, $\phi(x,p(x))$ is realisable. As $\phi$ is a $\Sigma_1$-formula, it follows with \Cref{Lemma: Sigma_1 realisability is truth} that $\VV \vDash \phi(x,p(x))$.
\end{proof}

Finally, we can use $\SRM^+$-realisability to easily determine the admissible rules of $\CZF$. A proof of Carl, Galeotti and Passmann \cite[Theorem 56]{CarlGaleottiPassmann2020} can be adapted to work here.

\begin{theorem}
    The propositional admissible rules of $\CZF$ are exactly the propositional admissible rules of intuitionistic logic.
\end{theorem}
\begin{proof}
    Using the fact that $\CZF$ is $\SRM^+$-realisable, we can prove this with glued realisability using \Cref{Theorem: Iemhoff}; almost exactly as we did in earlier joint work with \citet{CarlGaleottiPassmann2020}.
\end{proof}

\section{Beth Realisability Models}
\label{Section: Beth Realisability Models}

\subsection{Fallible Beth models}

In this section, we will make use of so-called \textit{fallible} Beth models because they satisfy a particular handy universal model theorem. 

\begin{definition}
    A \emph{fallible Beth frame} $(P,U)$ consists of a tree $P$ and an upwards closed set $U \subseteq P$ such that if every path through $p \in P$ meets $U$, then $p \in U$.
\end{definition}

\begin{definition}[Fallible Beth model]
    A \textit{fallible Beth model} $(P,U,D,I)$ for first-order logic consists of a fallible Beth tree $(P,U)$, domains $D_p$ for $p \in P$, and an interpretation $I_p$ of the language of first-order logic for each $p \in P$ such that:
    \begin{enumerate}
        \item $I_v(R) \subseteq I_w(R)$ for all $w \geq v$,
        \item $I_v(R) = D_v$ for all $v \in U$, and,
        \item if $R$ is an $n$-ary relation symbol, $\bar x \in D_v^n$ and on every path through $v$ there is some $w$ such that $\bar x \in I_w(R)$, then $\bar x \in I_v(R)$.
    \end{enumerate}
    A \textit{Beth model} is a fallible Beth model where $U = \emptyset$. If $p \in P$, then a \textit{bar for $p$} is a set $B \subseteq P$ such that every path through $p$ meets $B$. A \textit{$U$-bar for $p$} is a set $B \subseteq P$ such that $B \cup U$ is a bar for $p$. 
\end{definition}

\begin{definition}
    Let $(P,U,D,I)$ be a fallible Beth model and $v \in P$. We define by recursion on sentences in the language of first-order logic:
    \begin{enumerate}
        \item $v \Vdash \bot$ if and only if $v \in U$;
        \item $v \Vdash R(d_1,\dots,d_n)$ if and only if $(d_1,\dots,d_n) \in I_v(R)$;
        \item $v \Vdash A_0 \wedge A_1$ if and only if $v \Vdash A_0$ and $v \Vdash A_1$;
        \item $v \Vdash A_0 \vee A_1$ if and only if there is a bar $B$ for $v$ such that for every $w \in B$, $w \Vdash A_0$ or $w \Vdash A_1$;
        \item $v \Vdash A_0 \rightarrow A_1$ if and only if for every $w \geq v$, if $w \Vdash A_0$, then $w \Vdash A_1$;
        \item $v \Vdash \exists x A(x)$ if and only if there is a bar $B$ for $v$ such that for all $w \in B$, there is some $a \in D_w$ with $w \Vdash A(a)$;
        \item $v \Vdash \forall x A(x)$ if and only if for every $w \geq v$ and $a \in D_w$, $w \Vdash A(a)$.
    \end{enumerate}
\end{definition}

Note that, by this definition, if $v \in U$, then $v$ forces every formula trivially, i.e. the relation $\Vdash$ trivialises in $U$. By definition of $U$, it follows that if $v \notin U$ and $B$ is a $U$-bar for $v$, then $B \setminus U$ is non-empty. The following result of \citet[Chapter 13, Remark 2.6 and Theorem 2.8]{TroelstraVanDalen1988II} will be a crucial ingredient of our proof. 

\begin{theorem}
    \label{Theorem: Canonical fallible Beth model}
    Let $J$ be a recursively enumerable theory in intuitionistic first-order logic. Then there is a fallible Beth model $\mathcal{B}_J$ with constant domain $\omega$, based on the full binary tree of height $\omega$, such that $\mathcal{B} \Vdash A$ if and only if $J \vdash A$ for every sentence $A$ of first-order logic.
\end{theorem}

In what follows, we will refer to $\mathcal{B}_J$ as the \textit{universal Beth model for $J$}.

\subsection{Beth realisability models}

Inspired by \citet{vanOosten1991}, we now combine our notion of $\SRM^{+,O}_\LL$-realisability with Beth semantics. To make coherent use of oracles, we need the following definition.

\begin{definition}
    Let $P$ be a partial order. A \textit{system of oracles} $(O_v)_{v \in P}$ consists of partial class functions $O_v: \VV \to \VV$ such that, for all $w \geq v$, we have that $\dom(O_v) \subseteq \dom(O_w)$ and $O_v(x) = O_w(x)$ for all $x \in \dom(O_v)$.
\end{definition}

We need some notation to work with oracles. Given an $\SRM^{+,O}_\mathrm{L}$-program $r$, we write $r(x_1,\dots,x_n;O)$ for the result of the successful computation (if it exists) of $r$ on input $x_1,\dots,x_n$ and oracle $O$. If we work with a system of oracles $(O_v)_{v \in P}$, we also write $r(x_1,\dots,x_n;v)$ to mean $r(x_1,\dots,x_n;O_v)$. Finally, we write $r(x_1,\dots,x_n)$ to mean $r(x_1,\dots,x_n;\emptyset)$, i.e. the output (if it exists) of $r$ run with the empty oracle.

\begin{definition}
    Let $(P,U)$ be a fallible Beth frame, $(O_v)_{v \in P}$ be a system of oracles. We define recursively for sentences $\phi$ and $\psi$ in the language of set theory, for $a,b \in \mathrm{L}$, $v \in P$ and an $\SRM^{+,O}_\mathrm{L}$-program $r$:
    \begin{enumerate}
        \item $r \Vdash_v \bot$ if and only if $v \in U$;
        \item $r \Vdash_v a = b$ if and only if $a = b$ or $v \in U$;
        \item $r \Vdash_v a \in b$ if and only if $a \in b$ or $v \in U$;
        \item $r \Vdash_v \phi \wedge \psi$ if and only if $r(0;v) \Vdash_v \phi$ and $r(1;v) \Vdash_v \psi$;
        \item $r \Vdash_v \phi \vee \psi$ if and only if there is a $U$-bar $B$ for $v$ such that, for every $w \in B$, either $r(0;w) = 0$ and $r(1;w) \Vdash_w \phi$, or $r(0;w) = 1$ and $r(1;w) \Vdash \psi$;
        \item $r \Vdash_v \phi \rightarrow \psi$ if and only if for every $w \geq v$, if $s \Vdash_w \phi$, then $r(s;w) \Vdash_w \psi$;
        \item $r \Vdash_v \exists x \phi(x)$ if and only if there is a $U$-bar $B$ for $v$ such that for all $w \in B$, $r(1;w) \Vdash_w \phi(r(0;w))$;
        \item $r \Vdash_v \forall x \phi(x)$ if and only if for every $a$, $r(a;v) \Vdash_v \phi(a)$.
    \end{enumerate}
\end{definition}

If $v \in U$, then $r \Vdash_v \phi$ for every realiser $r$ and set-theoretic sentence $\phi$. The following is established by a standard argument.

\begin{theorem}
    Beth-realisability is sound for the axioms and rules of intuitionistic first-order logic.
\end{theorem}

\begin{lemma}
    \label{Lemma: Sigma_1 Beth-realisability is truth in L}
    Let $\phi(\bar x)$ be a $\Sigma_1$-formula and $v \notin U$. Then there is some realiser $r \Vdash_v \phi(\bar x)$ if and only if $L \vDash \phi(\bar x)$.
\end{lemma}
\begin{proof}
    As $v \notin U$, we know that any $U$-bar $B$ for $v$ satisfies $B \setminus U \neq \emptyset$. We prove this by induction. The cases for equality and set-membership are trivial.
    
    Suppose that $\Vdash_v \phi(\bar a) \wedge \psi(\bar a)$. By definition, this is equivalent to $\Vdash_v \phi(\bar a)$ and $\Vdash_v \psi(\bar a)$. Applying the induction hypothesis, this holds if and only if $\LL \vDash \phi(\bar a)$ and $\LL \vDash \psi(\bar a)$. This is, of course, equivalent to $\LL \vDash \phi(\bar a) \wedge \psi(\bar a)$.
    
    For disjunction, first suppose that $r \Vdash_v \phi(\bar a) \vee \psi(\bar a)$. By definition, this means that there is a $U$-bar $B$ for $v$ such that for all $w \in B$ we have either $r(0;w) = 0$ and $r(1;w) \Vdash_w \phi(\bar a)$, or $r(0;w) = 1$ and $r(1;w) \Vdash_w \psi(\bar a)$. Recall that $B \setminus U$ is non-empty. So take any $w \in B \setminus U$, then $\Vdash_w \phi(\bar a)$ or $\Vdash_w \psi(\bar a)$. By induction hypothesis, $\LL \vDash \phi(\bar a)$ or $\LL \vDash \psi(\bar a)$. Hence $\LL \vDash \phi(\bar a ) \vee \psi(\bar a)$.
    Conversely, assume that $\LL \vDash \phi(\bar a) \vee \psi(\bar a)$. Then $\LL \vDash \phi(\bar a)$ or $\LL \vDash \psi(\bar a)$. It follows, by induction hypothesis, that $\Vdash_v \phi(\bar a)$ or $\Vdash_v \psi(\bar a)$, but then $\Vdash_v \phi(\bar a) \vee \psi(\bar a)$.
    
    For implication, assume that $r \Vdash_v \phi \rightarrow \psi$. If $\LL \not \vDash \phi$, then trivially $\LL \vDash \phi \rightarrow \psi$. So assume that $\LL \vDash \phi$. By induction hypothesis, we know that there is a realiser $s \Vdash_v \phi$. Hence, $r(s) \Vdash_v \psi$. Applying the induction hypothesis once more, we get $\LL \vDash \psi$. Conversely, assume that $\LL \vDash \phi \rightarrow \psi$. If $\LL \not \vDash \phi$, then, by induction hypothesis, $\not \Vdash_w \phi$ for all $w \geq v$. So $\Vdash_v \phi \rightarrow \psi$ holds trivially. If $\LL \vDash \phi$, then $\LL \vDash \psi$. So, by induction hypothesis, there is a realiser $s \Vdash_v \psi$. Hence, a realiser for $\phi \rightarrow \psi$ is the $\SRM$ $p$ that returns $s$ on any input.
    
    For bounded universal quantification, assume that $\LL \vDash \forall x \in y \phi(x)$. Then, by induction hypothesis, we can find a function $f: y \to \LL$ such that $f(z) \Vdash_v \phi(z)$. Let $p$ be the $\SRM$ with parameter $f$ that returns $f(z)$ on input $z$. Then $p \Vdash_v \forall x \in y \phi(x)$. Conversely, note that $\Vdash_v \forall x \in y \phi(x)$ entails that $\Vdash_v \phi(x)$ for every $x \in y$. An application of the induction hypothesis yields $\LL \vDash \forall x \in y \phi(x)$.
    
    For unbounded existential quantification, assume that $\LL \vDash \exists x \phi(x)$. Then there is some $a \in \LL$ such that $\LL \vDash \phi(a)$. By induction hypothesis, there is a realiser $s \Vdash_v \phi(a)$. Let $p$ be an $\SRM$ such that $p(1) = s$ and $p(0) = a$ (by using, if necessary, parameter $a$). Then $p \Vdash_v \exists x \phi(x)$. Conversely, if $p \Vdash_v \exists x \phi(x)$, then there is a $U$-bar $B$ for $v$ such that for all $w \in B$, $p(1;w) \Vdash_w \phi(p(0;w))$. Take any $w \in B$ and the induction hypothesis implies that $\LL \vDash \phi(p(0;w))$, and, hence, $\LL \vDash \exists x \phi(x)$.
\end{proof}

\begin{theorem}
    The Beth realisability model satisfies $\CZF + \mathrm{PowerSet} + \mathrm{AC}$.
\end{theorem}
\begin{proof}
    Realisers for the axioms and schemas can be constructed (almost exactly) as in the proof of \Cref{Theorem: CZF realisable}. For the case of $\Delta_0$-separation, observe that the use of \Cref{Lemma: Sigma_1 realisability is truth} has to be replaced with \Cref{Lemma: Sigma_1 Beth-realisability is truth in L}. (Note that we only need to consider the cases for $v \notin U$, as the other case is trivial.)
\end{proof}

\subsection{Constructing a model for a given logic}

The goal of this section is to construct a Beth-realisability model that matches the truth in the universal Beth model $\mathcal{B}_J = (P,U, D,I)$ for a given logic $J$. To begin with, we define the two systems of oracles $(F_v)_{v \in P}$ and $(G_v)_{v \in P}$. If $a$ is a set, let $\rank_{\omega}(a)$ be the unique natural number such that $\rank(a) = \alpha + \rank_{\omega}(a)$ for a maximal (possibly $0$) limit ordinal $\alpha$.
\begin{enumerate}
    \item We define $F_v: \mathrm{V} \times \mathrm{V} \to \mathrm{V}$ by recursion on $v \in P$ ($P$ being the binary tree of height $\omega$) such that:
    $$
        F_v(m,\seq{b_0,\dots,b_n}) = \begin{cases}
            a, 
                & \text{if } m = \godel{\exists x A(x,y_0,\dots,y_n)} \\ &\phantom{if} \text{and } w \leq v \text{ is least such that }
                 a \in \omega \text{ is least with } \\ &\phantom{if} \mathcal{B}_J, w \Vdash A(a,\rank_\omega(b_0),\dots,\rank_\omega(b_n)), \\
            i, 
                & \text{if } m = \godel{(A_0 \vee A_1)(y_0,\dots,y_n)} \\ &\phantom{if} \text{and }
                 w \leq v \text{ is least such that } i \in \omega \text{ is least with } 
                \\ &\phantom{if} \mathcal{B}_J, w \Vdash A_i(\rank_\omega(b_0),\dots,\rank_\omega(b_n)), \\
            \text{undefined}, 
                & \text{otherwise.}
        \end{cases}
    $$
    \item We define $G_v: \mathrm{V} \times \mathrm{V} \to \mathrm{V}$ such that:
    $$
    G_v(a,b) = \begin{cases}
        1, 
            & \text{if } b = \seq{i,b_0,\dots,b_n}, 
                \\& \phantom{if} \ \mathcal{B}_J, v \Vdash P_i(\rank_\omega(b_0),\dots,\rank_\omega(b_n)),
                \\& \phantom{if} \text{ and } D(a,b) = 1, \\
        0, 
            & \text{if } b = \seq{i,b_0,\dots,b_n},
                \\& \phantom{if} \ \mathcal{B}_J, v \Vdash P_i(\rank_\omega(b_0),\dots,\rank_\omega(b_n)),
                \\& \phantom{if} \text{ and } D(a,b) = 0, \\
        \text{undefined,} 
            & \text{otherwise.}
    \end{cases}
    $$
\end{enumerate}

\begin{lemma}
    The sequences $(F_v)_{v \in P}$ and $(G_v)_{v \in P}$ form systems of oracles. \qed
\end{lemma}

From now on, we consider the Beth-realisability based on these systems of oracles. Note that, without loss of generality, we can combine two systems of oracles into one by, e.g., taking $O_v(\seq{0,x}) = F_v(x)$ and $O_v(\seq{1,x}) = G_v(x)$ for all $v \in P$.

\begin{lemma}
    \label{Lemma: negative formula for D}
    Let $v \notin U$. There is a negative formula $\psi(x,y)$ such that there is a realiser $r \Vdash_v \psi(x,y)$ is realised if and only if $D(x,y) = 1$.
\end{lemma}
\begin{proof}
    Except for the power set case, every clause of the definition of successful computation (\Cref{Definition: SRM computation}), adapted for $\SRM^+_\LL$, can be written as a $\Sigma_1$-formula. For the $\TAKE$-operation, recall that $<_\LL$ is $\Sigma_1$-definable. Now consider the predicate ``$x = \Pow(y)$'' which is needed for the $\POW$-operation and can be formalised as ``$\forall z (z \in x \leftrightarrow \forall w \in z \, w \in y)$''. As the part in brackets is a $\Delta_0$-formula, it follows with \Cref{Lemma: Sigma_1 Beth-realisability is truth in L} that this predicate is realised if and only if it is true. Note, in particular, that also the successor case for the halting problem oracle is realised if and only if it is true in $\LL$. This is because the existence of a successful computation is absolute, as we have just seen. 

    Applying \Cref{Lemma: Sigma_1 Beth-realisability is truth in L} once more, these observations show that we can construct a formula $\chi$ expressing ``$c$ is a successful computation of $D(x,y)$ with result $0$'' such that $\chi(c,x,y)$ is realised if and only if it is true in $\LL$. Take $\psi(x,y)$ to be $\neg \exists c \chi(c,x,y)$. It follows that $\psi(x,y)$ is realised if and only if $D(x,y) = 1$ because $D$ halts on every input with either $0$ or $1$ as output.
\end{proof}

\begin{lemma}
    \label{Lemma: Translation Base Case}
    Let $P_i(y_0,\dots,y_n)$ be a predicate in the language of first-order logic. There is a set-theoretic formula $\phi_i(y_0,\dots,y_n)$ and a realiser $r$ such that for all $b_0,\dots,b_n \in \mathrm{L}$, $r(b_0,\dots,b_n) \Vdash_v \phi_i(b_0,\dots,b_n)$ if and only if $\mathcal{B}_J, v \Vdash P_i(\rank_{\omega}(b_0),\dots,\rank_{\omega}(b_n))$.
\end{lemma}
\begin{proof}
    Let $\psi(x,y)$ be the negative formula from \Cref{Lemma: negative formula for D} expressing that $D(x,y) = 1$. As $\psi$ is negative, we know that, for every $v$ and $a, b \in \mathrm{L}$, either $\Vdash_v \psi(a,b)$ or $\Vdash_v \neg \psi(a,b)$. Then take:
    $$
        \phi_i(y_0, \dots, y_n) \quad := \quad \forall x (\psi(x,\seq{i,y_0,\dots,y_n}) \vee \neg \psi(x,\seq{i,y_0,\dots,y_n})).
    $$
    Suppose there was a realiser $r \Vdash_v \phi_i(b_0,\dots,b_n)$ but $\mathcal{B}_J, v \not \Vdash P_i(\rank_\omega(b_0),\dots,\rank_\omega(b_n))$. In this situation, we can decide $D_{\seq{i,b_0,\dots,b_n}}$ from $r$ for every $a$: if $r(a,b_0,\dots,b_n)$ returns a realiser for $\psi(a,\seq{i,b_0,\dots,b_n})$, then $D_{\seq{i,b_0,\dots,b_n}}(a) = 1$; if $r(a,b_0,\dots,b_n)$ returns a realiser for $\neg \psi(a,\seq{i,b_0,\dots,b_n})$, then $D_{\seq{i,b_0,\dots,b_n}}(a) = 0$. However, by our assumption, $G_v(c,\seq{i,b_0,\dots,b_n})$ is undefined for all $c \in \mathrm{L}$. This means that $r$ cannot query the oracle $G_v$ on elements of the form $(c,\seq{i,b_0,\dots,b_n})$ because then the computation would not be successful. Hence, using $r$, we can construct a witnesses that $D_{\seq{i,b_0,\dots,b_n}}$ is computable in $D^{\seq{i,b_0,\dots,b_n}}$ but that is a contradiction to \Cref{Theorem: D for SRM^+_L}. (Note that $F$ does not matter here because the information contained in $F$ could be saved in a set-sized parameter.)
    
    Conversely, assume that $\mathcal{B}_J, v \Vdash P_i(\rank_\omega(b_0),\dots,\rank_\omega(b_n))$. By definition of $G$, it follows that $G_v(a, \seq{i,b_0,\dots,b_n})$ is defined for all $a \in \mathrm{L}$. Hence, a realiser for $\phi_i$ can be easily obtained by querying the oracle $G(a, \seq{i,b_0,\dots,b_n})$: if the result is $1$, then return a realiser of $\psi(a, \seq{i,b_0,\dots,b_n})$. If the result is $0$, then return a realiser of $\neg \psi(a,\seq{i,b_0,\dots,b_n})$. In both cases, the computation of the corresponding realiser is trivial because the formulas are negative. 
\end{proof}

Let $\tau(P_i) = \phi_i$ and extend $\tau$ to a translation of all formulas in the language of first-order logic in the obvious way. Note that the formulas $\phi_i$ are $\Pi_3$-formulas.

\begin{lemma}
    \label{Lemma: Translation}
    Let $A(y_0,\dots,y_n)$ be a formula in the language of first-order logic. Then:
    \begin{enumerate}
        \item If there is a realiser $r \Vdash_v A^\tau(b_0,\dots,b_n)$, then $\mathcal{B}_J,v \Vdash A(\rank_\omega(b_0),\dots,\rank_\omega(b_n))$.
        \item There is a realiser $r_A$ such that for all $b_0,\dots,b_n \in \mathrm{L}$, if $\mathcal{B}_J, v \Vdash A(\rank_\omega(b_0),\dots,\rank_\omega(b_n))$, then $r_A(b_0,\dots,b_n) \Vdash_v A^\tau(b_0,\dots,b_n)$.
    \end{enumerate}
\end{lemma}
\begin{proof}
    We prove (i) and (ii) simultaneously by induction so that both directions are available in the induction hypothesis. We begin with proving the cases for (i). The base case follows from \Cref{Lemma: Translation Base Case}. For conjunction, $A \wedge B$, note that $\Vdash_v A^\tau \wedge B^\tau$ entails $\Vdash_v A^\tau$ and $\Vdash_v B^\tau$. Hence, by induction hypothesis, $\mathcal{B}_J, v \Vdash A$ and $\mathcal{B}_J, v \Vdash B$. So, $\mathcal{B}_J, v \Vdash A \land B$. For disjunction, $A \vee B$, we have that $r \Vdash_v A^\tau \lor B^\tau$ entails that there is a $U$-bar $B$ for $v$ such that for every $w \in B$, either $r^w(0) = 0$ and $r^w(1) \Vdash_w A^\tau$ or $r^w(0) = 1$ and $r^w(1) \Vdash_w B^\tau$. By induction hypothesis, this means that there is a $U$-bar $B$ for $v$ such that for every $w \in B$, $w \Vdash A$ or $w \Vdash B$. Hence $v \Vdash A \lor B$. The case for implication is similar (making use of (ii) as well), and the cases for universal and existential quantification follow with the induction hypothesis.
    
    For the cases for (ii), we recursively construct the required realisers $r_A(b_0,\dots,b_n)$, uniform in $b_0,\dots,b_n \in \mathrm{L}$, for each formula $A$. Once more, the base case, $r_{P_i}(y_0,\dots,y_n)$, was established in \Cref{Lemma: Translation Base Case}. To keep notation light, we will write $\bar y$ for $y_0,\dots,y_n$ (or, potentially, a subsequence of this), and similarly for $\bar b$.
    
    For conjunction $(A \land B)(\bar y)$, take $r_{(A \land B)(\bar y)}(\bar b)(0) = r_A(\bar b)$ and $r_{(A \land B)(\bar y)}(\bar b)(1) = r_B(\bar b)$. An application of the induction hypothesis shows that $r_{(A \land B)(\bar y)}$ does the job. 
    
    For implication $(A \rightarrow B)(\bar y)$, we know by our induction hypothesis---for both (i) and (ii)---that $r_{B(\bar y)}(\bar b) \Vdash_w B(\bar b)$ if and only if $w \Vdash B(\bar b)$ for all $w \geq v$. Hence, let $r_{A \rightarrow B(\bar y)}(\bar b,s) = r_{B}(\bar b)$. It is straightforward to check that this does the job.

    For disjunction, define $r_{A \vee B}(\bar y)$ to be the $\SRM^{+,O}$ that, on input $\bar b$, returns a code $s$ for an $\SRM^{+,O}$ with parameters $\bar b$ that does the following. On input $0$, $s$ calls the oracle $F$ on $(\godel{(A \vee B)(\bar y)}, \seq{\bar b})$ and returns this value. On input $1$, $s$ returns $r_A(\bar b)$ if $F(\godel{(A \vee B)(\bar y)}, \seq{\bar b}) = 0$ and it returns $r_B(\bar b)$ otherwise. To see that $r_{(A \vee B)(\bar y)}$ does the job, assume that there is a $U$-bar $B$ such that for every $w \in B$, $w \Vdash A(\bar b)$ or $w \Vdash B(\bar b)$. Equivalently, by induction hypothesis, for every $w \in B$, $r_A(\bar b;w) \Vdash_w A(\rank_\omega(b_0),\dots,\rank_\omega(b_n))$ or $r_B(\bar b;w) \Vdash_w B(\rank_\omega(b_0),\dots,\rank_\omega(b_n))$. By definition of $r_{(A \vee B)(\bar y)}$, it follows that $r_{(A \vee B)(\bar y)}(\bar b;w)(0) = r_A$ or $r_{(A \vee B)(\bar y)}(\bar b;w)(0) = r_B$. Hence, $r_{(A \vee B)(\bar y)}(\bar b;w) \Vdash_w (A \vee B)(\bar b)$.
    
    For existential quantification, define $r_{\exists x A(x,\bar y)}$ to be the function that, on input $\bar b$, calls the oracle $F$ on input $(\godel{\exists x A(x,y)},\seq{\bar b})$. Let the result of this query be $n \in \omega$. Then let $r_{\exists x A(x,\bar y)}(0) = n$ and $r_{\exists x A(x,\bar y)}(1) = r_{A(n,\bar y)}$. Note here that we do not require the use of parameters because the realiser $r_{A(n,\bar y)}$ is uniform in $n,\bar y$. To check that $r_{\exists x A(x,\bar y)}$ does the job, let $\bar b \in \mathrm{L}$ and assume that there is a $U$-bar $B$ for $v$ such that, for every $w \in B$, there is some $n_w \in \omega$ such that $w \Vdash A(n_w, \rank_\omega(\bar b))$. By induction hypothesis, it follows that $r_{A(n_w,\bar y)}(\bar b;w) \Vdash_w A^\tau(n_w,\bar b)$ (as $\mathcal{B}_J$ has constant domain $\omega$ and $\rank_\omega(n_w) = n_w$), i.e. $r_{\exists x A(x,\bar y)}(\bar b, 0; w) \Vdash_w A^\tau(r_{\exists x A(x,\bar y)}(\bar b, 1; w),\bar b)$. Hence, $r_{\exists x A(x,\bar y)}(\bar b; v) \Vdash_v \exists x A^\tau(x, \bar b)$.
    
    For universal quantification, define $r_{\forall x A(x,\bar y)}(\bar y)$ to be the function that returns $r_{A(x,\bar y)}(x,\bar y)$.
\end{proof}

If $J$ is a set of formulas in first-order logic, we write $J^\tau$ for the image of $J$ under $\tau$ (i.e. $J^\tau = \tau[J]$).

\begin{theorem}
    \label{Theorem: Main result}
    Let $J$ be a recursively enumerable theory in intuitionistic first-order logic, and $T \subseteq \CZF + \mathrm{PowerSet} + \mathrm{AC}$.
    Then $T + J^\tau \vdash A^\tau$ if and only if $J \vdash_\IQC A$.
\end{theorem}
\begin{proof}
    The backwards direction is straightforward with the soundness of the Beth realisability model. For the forward direction, assume that $J \not \vdash A$. Then, by \Cref{Theorem: Canonical fallible Beth model}, we know that $\mathcal{B}_J \not \vdash A$. In this situation, \Cref{Lemma: Translation} implies that there is no realiser of $A^\tau$. But the same lemma implies that $B^\tau$ is realised for every $B \in J$. Hence, $T + J^\tau \not \vdash A^\tau$. 
\end{proof}

The following corollary follows immediately by taking $J = \emptyset$.

\begin{corollary}
    \label{Theorem: FOL of CZF is intuitionistic}
    Let $T \subseteq \CZF + \mathrm{PowerSet} + \mathrm{AC}$ be a set theory. Then the first-order logic of $T$ is intuitionistic first-order logic, $\QL(T) = \IQC$. In particular, $\QL(\CZF) = \IQC$.
\end{corollary}

\begin{remark}
    \citet{Rathjen2002} points out that ``the combination of $\CZF$ and the general axiom of choice has no constructive justification in Martin-Löf type theory''. In contrast, our results show that the combination of $\CZF$ and the axiom of choice is innocent \textit{on a logical level} in that adding the axiom of choice does not result in an increase of logical strength: $\QL(\CZF + \mathsf{AC}) = \QL(\CZF) = \IQC$. Note, of course, that $\CZF + \AC$ satisfies the law of excluded middle for $\Delta_0$-formulas. This follows from the proof of Diaconescu's theorem (see \cref{Section: Preliminaries}) which only requires $\Delta_0$-separation to prove the law of excluded middle for $\Delta_0$-formulas. Such theories satisfying the law of excluded middle for $\Delta_0$-formulas but not in general are sometimes called semi-intuitionistic.
\end{remark}

\medskip

\subsection*{Acknowledgements} I am thankful for the very helpful remarks of an anonymous reviewer. Moreover, I would like to thank Merlin Carl, Lorenzo Galeotti, Benedikt Löwe, Benno van den Berg and Ned Wontner for helpful discussions. I thank Daniël Otten for spotting a few typos. 

\subsection*{Funding} This research was supported by a doctoral scholarship of the \emph{Studienstiftung des deutschen Volkes} (German Academic Scholarship Foundation).


\bibliographystyle{plainnat}
\bibliography{references}

\begin{thebibliography}{25}
\providecommand{\natexlab}[1]{#1}
\providecommand{\url}[1]{\texttt{#1}}
\expandafter\ifx\csname urlstyle\endcsname\relax
  \providecommand{\doi}[1]{doi: #1}\else
  \providecommand{\doi}{doi: \begingroup \urlstyle{rm}\Url}\fi

\bibitem[Aczel and Rathjen(2001)]{AczelRathjen2001}
Peter Aczel and Michael Rathjen.
\newblock Notes on constructive set theory, 2001.

\bibitem[Carl(2020)]{Carl2020}
Merlin Carl.
\newblock \emph{Ordinal computability}, volume~9 of \emph{De Gruyter Series in
  Logic and its Applications}.
\newblock De Gruyter, Berlin, 2020.
\newblock ISBN 978-3-11-049562-1; 978-3-11-049615-4; 978-3-11-049291-0.
\newblock \doi{10.1515/9783110496154}.
\newblock URL \url{https://doi.org/10.1515/9783110496154}.
\newblock An introduction to infinitary machines.

\bibitem[Carl et~al.(2020)Carl, Galeotti, and
  Passmann]{CarlGaleottiPassmann2020}
Merlin Carl, Lorenzo Galeotti, and Robert Passmann.
\newblock Realisability for infinitary intuitionistic set theory, 2020.

\bibitem[de~Jongh(1968)]{DeJonghUnpublished}
Dick de~Jongh.
\newblock The maximality of the intuitionistic predicate calculus with respect
  to heyting's arithmetic.
\newblock Unpublished article with abstract appearing in \cite{DeJongh1970},
  1968.

\bibitem[de~Jongh(1970)]{DeJongh1970}
Dick de~Jongh.
\newblock The maximality of the intuitionistic predicate calculus with respect
  to {H}eyting's arithmetic.
\newblock \emph{The Journal of Symbolic Logic}, 35\penalty0 (4):\penalty0 606,
  1970.

\bibitem[de~Jongh et~al.(2011)de~Jongh, Verbrugge, and
  Visser]{deJonghVerbruggeVisser2011}
Dick de~Jongh, Rineke Verbrugge, and Albert Visser.
\newblock Intermediate logics and the de {J}ongh property.
\newblock \emph{Archive for Mathematical Logic}, 50\penalty0 (1):\penalty0
  197--213, Feb 2011.

\bibitem[Diaconescu(1975)]{Diaconescu1975}
Radu Diaconescu.
\newblock Axiom of choice and complementation.
\newblock \emph{Proceedings of the American Mathematical Society}, 51:\penalty0
  176--178, 1975.
\newblock ISSN 0002-9939.
\newblock \doi{10.2307/2039868}.
\newblock URL \url{https://doi-org.proxy.uba.uva.nl:2443/10.2307/2039868}.

\bibitem[Fraenkel et~al.(1973)Fraenkel, Bar-Hillel, and Levy]{Fraenkel1973}
Abraham~Adolf Fraenkel, Yehoshua Bar-Hillel, and Azriel Levy.
\newblock \emph{Foundations of set theory}.
\newblock Elsevier, 1973.

\bibitem[Friedman and Ščedrov(1986)]{FriedmanScedrov1986}
Harvey~M. Friedman and Andrej Ščedrov.
\newblock On the quantificational logic of intuitionistic set theory.
\newblock \emph{Mathematical Proceedings of the Cambridge Philosophical
  Society}, 99\penalty0 (1):\penalty0 5--10, 1986.
\newblock ISSN 0305-0041.
\newblock \doi{10.1017/S0305004100063854}.
\newblock URL \url{https://doi.org/10.1017/S0305004100063854}.

\bibitem[Goodman and Myhill(1978)]{GoodmanMyhill1978}
N~Goodman and J~Myhill.
\newblock Choice implies excluded middle.
\newblock \emph{Mathematical Logic Quarterly}, 24\penalty0 (25-30):\penalty0
  461--461, 1978.

\bibitem[Iemhoff(2001)]{Iemhoff2001}
Rosalie Iemhoff.
\newblock On the admissible rules of intuitionistic propositional logic.
\newblock \emph{The Journal of Symbolic Logic}, 66\penalty0 (1):\penalty0
  281--294, 2001.
\newblock ISSN 00224812.
\newblock URL \url{http://www.jstor.org/stable/2694922}.

\bibitem[Iemhoff(2005)]{Iemhoff2005}
Rosalie Iemhoff.
\newblock Intermediate logics and {V}isser's rules.
\newblock \emph{Notre Dame J. Formal Logic}, 46\penalty0 (1):\penalty0 65--81,
  01 2005.
\newblock \doi{10.1305/ndjfl/1107220674}.
\newblock URL \url{https://doi.org/10.1305/ndjfl/1107220674}.

\bibitem[Iemhoff and Passmann(2019)]{IemhoffPassmannUnpublished2019}
Rosalie Iemhoff and Robert Passmann.
\newblock Notes on admissible rules for constructive set theories.
\newblock Draft, 2019.

\bibitem[Iemhoff and Passmann(2021)]{IemhoffPassmann2021}
Rosalie Iemhoff and Robert Passmann.
\newblock Logics of intuitionistic {K}ripke-{P}latek set theory.
\newblock \emph{Ann. Pure Appl. Logic}, 172\penalty0 (10):\penalty0 Paper No.
  103014, 22, 2021.
\newblock ISSN 0168-0072.
\newblock \doi{10.1016/j.apal.2021.103014}.
\newblock URL \url{https://doi.org/10.1016/j.apal.2021.103014}.

\bibitem[Jech(2003)]{Jech2003}
Thomas Jech.
\newblock \emph{Set theory}.
\newblock Springer Monographs in Mathematics. Springer-Verlag, Berlin, 2003.
\newblock ISBN 3-540-44085-2.
\newblock The third millennium edition, revised and expanded.

\bibitem[Kleene and Post(1954)]{KleenePost1954}
S.~C. Kleene and Emil~L. Post.
\newblock The upper semi-lattice of degrees of recursive unsolvability.
\newblock \emph{Ann. of Math. (2)}, 59:\penalty0 379--407, 1954.
\newblock ISSN 0003-486X.
\newblock \doi{10.2307/1969708}.
\newblock URL \url{https://doi.org/10.2307/1969708}.

\bibitem[Leivant(1979)]{Leivant1979}
Daniel Leivant.
\newblock \emph{Absoluteness of Intuitionistic Logic}, volume~13 of \emph{ILLC
  Historical Dissertations Series (HDS)}.
\newblock Institute for Logic, Language and Computation, University of
  Amsterdam, 1979.

\bibitem[Passmann(2020)]{Passmann2020}
Robert Passmann.
\newblock {D}e {J}ongh's theorem for intuitionistic {Z}ermelo-{F}raenkel set
  theory.
\newblock In Maribel Fern{\'{a}}ndez and Anca Muscholl, editors, \emph{28th
  {EACSL} Annual Conference on Computer Science Logic, {CSL} 2020, January
  13-16, 2020, Barcelona, Spain}, volume 152 of \emph{LIPIcs}, pages
  33:1--33:16. Schloss Dagstuhl - Leibniz-Zentrum f{\"{u}}r Informatik, 2020.
\newblock \doi{10.4230/LIPIcs.CSL.2020.33}.
\newblock URL \url{https://doi.org/10.4230/LIPIcs.CSL.2020.33}.

\bibitem[Rathjen(2002)]{Rathjen2002}
Michael Rathjen.
\newblock Choice principles in constructive and classical set theories.
\newblock In \emph{Logic Colloquium}, volume~2, pages 299--326. Cambridge
  University Press, 2002.

\bibitem[Rathjen(2006)]{Rathjen2006}
Michael Rathjen.
\newblock The formulae-as-classes interpretation of constructive set theory.
\newblock In \emph{Proof technology and computation}, volume 200 of \emph{NATO
  Sci. Ser. III Comput. Syst. Sci.}, pages 279--322. IOS, Amsterdam, 2006.

\bibitem[Tharp(1971)]{Tharp1971}
Leslie~H. Tharp.
\newblock A quasi-intuitionistic set theory.
\newblock \emph{J. Symbolic Logic}, 36:\penalty0 456--460, 1971.
\newblock ISSN 0022-4812.
\newblock \doi{10.2307/2269954}.
\newblock URL \url{https://doi-org.proxy.uba.uva.nl/10.2307/2269954}.

\bibitem[Troelstra and van Dalen(1988)]{TroelstraVanDalen1988II}
A.~S. Troelstra and D.~van Dalen.
\newblock \emph{Constructivism in mathematics. {V}ol. {II}}, volume 123 of
  \emph{Studies in Logic and the Foundations of Mathematics}.
\newblock North-Holland Publishing Co., Amsterdam, 1988.
\newblock ISBN 0-444-70358-6.
\newblock An introduction.

\bibitem[van~den Berg and Moerdijk(2012)]{vandenBergMoerdijk2012}
Benno van~den Berg and Ieke Moerdijk.
\newblock Derived rules for predicative set theory: an application of sheaves.
\newblock \emph{Annals of pure and applied logic}, 163\penalty0 (10):\penalty0
  1367--1383, 2012.

\bibitem[van Oosten(1991)]{vanOosten1991}
Jaap van Oosten.
\newblock A semantical proof of de {J}ongh's theorem.
\newblock \emph{Arch. Math. Logic}, 31\penalty0 (2):\penalty0 105--114, 1991.
\newblock ISSN 0933-5846.
\newblock \doi{10.1007/BF01387763}.
\newblock URL \url{https://doi.org/10.1007/BF01387763}.

\bibitem[Visser(1999)]{Visser1999}
Albert Visser.
\newblock Rules and arithmetics.
\newblock \emph{Notre Dame J. Formal Logic}, 40\penalty0 (1):\penalty0
  116--140, 1999.
\newblock ISSN 0029-4527.
\newblock \doi{10.1305/ndjfl/1039096308}.
\newblock URL \url{https://doi.org/10.1305/ndjfl/1039096308}.
\newblock Special issue in honor and memory of George S. Boolos (Notre Dame,
  IN, 1998).

\end{thebibliography}

\end{document}